\documentclass[11pt,a4paper]{article}

\usepackage{amsmath,amssymb}
\usepackage{mathrsfs,amsfonts}
\usepackage{bm}
\usepackage{graphicx}
\usepackage{epsfig}
\usepackage{indentfirst}
\usepackage{extarrows}
\usepackage{amsthm}

\begin{document}


\title{\textbf{Blow-up problems for nonlinear parabolic equations on locally finite graphs}}
\author{Yong Lin \quad Yiting Wu }

\date{}
\maketitle

\vspace{-12pt}


\begin{minipage}{145mm}
\noindent{\small\textbf{Abstract} \; Let $G=(V,E)$ be a locally finite connected weighted graph, $\Delta$ be the usual graph Laplacian.
In this paper,
we study the blow-up problems for the nonlinear parabolic equation $u_t=\Delta u + f(u)$
on $G$. The blow-up phenomenons of the equation are discussed in terms of two cases:
(i) an initial condition is given; (ii) a Dirichlet boundary condition is given.
We prove that if $f$ satisfies appropriate conditions, then the solution of the equation blows up in a finite time.
}
\end{minipage}

\vspace{12pt}

\section{Introduction and main results}

As is known to us, many structures in our real life can be represented by a connected  graph whose vertices represent nodes and whose edges represent their links, such as the internet, brain, organizations and so on.
In recent years, the investigations of discrete weighted Laplacians and various equations on graphs
have attracted attention from many authors
(see \cite{BHJ,CLC,DK,gly1,gly2,gly3,HAESELER,LCY,AW,XXM} and references therein).
There have been some works on dealing with blow-up phenomenon of the equations on graphs, for example, Xin et al. \cite{XXM} investigated the blow-up properties of the Dirichlet boundary value problem for $u_t=\Delta u+u^p$ $(p>0)$
on a finite graph.
However, as far as we know, the blow-up phenomenon on a locally finite graph
has not been studied in the literature.
The main concern of this paper is to discuss the blow-up phenomenon of
the nonlinear parabolic equation $u_t=\Delta u+f(u)$ on a locally finite graph.
This equation is the mathematical model of heat diffusion and can be used to model solid fuel ignition \cite{BD}.
The function $f(u)$ is typically a nonlinear function, such as $u^p (p>1)$.
The main purpose of the paper is to study  blow-up phenomenon of the nonlinear parabolic equation
in terms of the following two cases:
(i) an initial condition is given;
(ii) a Dirichlet boundary condition is given.

Let $G=(V,E)$ be a locally finite connected graph,
where $V$ denotes the vertex set and $E$ denotes the edge set.
For any $T>0$, a function $u=u(t,x)$ is said to be a solution of (1.1) in
$[0,T]\times V$, if the equation (1.1) is satisfied by $u$ in $[0,T]\times V$, meanwhile,
$u$ is bounded and continuous with respect to $t$ in $[0,T]\times V$.
Moreover, we say that a solution $u$ blows up in a finite time $T$, if there exists $x\in V$ such that
$$|u(t,x)|\rightarrow +\infty \quad \textrm{as} \quad t\rightarrow T^-.$$

Focussing on the research goals mentioned above,
in this paper, we first deal with the blow-up phenomenon of the following Cauchy problem on $G$
\begin{equation*}
\tag{1.1}
\begin{cases}
\frac{\partial}{\partial t}u(t,x)=\Delta u(t,x) + f(u(t,x)), &\quad (t,x)\in (0,+\infty)\times V,\\
u(0,x)=a(x), &\quad x\in V,
\end{cases}
\end{equation*}
where $\Delta$ is the $\mu-$Laplacian on $G$.

Next, given a non-empty finite subset $\Omega\subset V$, we denote the
Dirichlet Laplacian on $\Omega^\circ$ by $\Delta_\Omega$.
We consider the blow-up phenomenon arising from the following discrete nonlinear parabolic equations on $G$
\begin{equation*}
\tag{1.2}
\begin{cases}
\frac{\partial}{\partial t}u(t,x)=\Delta_\Omega u(t,x) + f(u(t,x)), &\quad (t,x)\in (0,+\infty)\times \Omega^\circ,\\
u(0,x)=a(x), &\quad x\in \Omega^\circ,\\
u(t,x)=0, &\quad (t,x)\in [0,+\infty)\times \partial\Omega,
\end{cases}
\end{equation*}
where $\partial\Omega$ is the boundary of $\Omega$,\,\,\,$\Omega^\circ$ is the interior of $\Omega$,
which are express by
\begin{center}
$\partial\Omega:=\{x\in \Omega: \,\,\textrm{$\exists$ $y\in \Omega^c$ such that $y$ is adjacent to $x$}\}$,\,\,
$\Omega^\circ:=\Omega \setminus \partial\Omega$.
\end{center}

The earlier blow-up results on parabolic equations on $\mathbb{R}$ are due to Kaplan \cite{K} and Fujita \cite{Fujita1,Fujita2}.
For the finite time blow-up, Osgood \cite{O} gave a criterion, namely, the nonlinear term on the right hand side of equation $u_t=\Delta u+f(u)$ must satisfy
$$\int^\infty_r\frac{d\tau}{f(\tau)}<\infty \,\,\,\textrm{ for some } r>0.$$

In this paper, we consider the blow-up problems for the nonlinear parabolic equation $u_t=\Delta u+f(u)$
on a locally finite graph $G$. We establish our results under assumptions that $f$ satisfies the following properties:\\
(H1) $f$ is continuous in $[0,+\infty)$;\\
(H2) $f(0)\geq 0$ and $f(\tau)>0$ for all $\tau>0$;\\
(H3) $f$ is convex in $[0,+\infty)$;\\
(H4) $\int^\infty_r\frac{d\tau}{f(\tau)}<\infty$ for some $r>0$.\\

Our main results are stated in Theorems 1.1 and 1.2 below.
\newtheorem{remark}{\textbf{Remark}}[section]
\newtheorem{corollary}{\textbf{Corollary}}[section]
\newtheorem{theorem}{\textbf{Theorem}}[section]
\begin{theorem}
\textnormal{Let $G$ be a locally finite connected graph and have the volume growth of positive degree $m$.
Suppose $f$ satisfies the assumptions (H1)-(H4) with $f(0)=0$, and $a(x)$ given by (1.1)
is bounded, non-negative and not trivial in $V$.
Set $$F(r)=\int_r^{+\infty} \frac{d\tau}{f(\tau)}.$$
If there exists a $\theta\in (0,1)$ such that
\begin{equation*}
\tag{1.3}
F(t^{-1})\leq t^{\frac{\theta}{m}}
\end{equation*}
for sufficiently large $t$,
then the non-negative solution $u$ of (1.1) blows up in a finite time.}
\end{theorem}


\begin{theorem}
\textnormal{Let $G$ be a locally finite connected graph.
Suppose $f$ satisfies the assumptions (H1)-(H4) with $f\in C^1[0,+\infty)$, and $a(x)$ given by (1.2)
is bounded, non-negative and not trivial in $\Omega^\circ$.
If $f(\tau)-\lambda_1 \tau>0$ for $\tau>\kappa$,
then the solution $u$ of (1.2) blows up in a finite time,
where
$\kappa=\int_{\Omega^\circ}a(x)\phi_1(x)>0$,
$\lambda_1=\lambda_1(\Omega)$ is the first eigenvalue of $-\Delta_\Omega$,
$\phi_1(x)$ is the eigenfunction corresponding to $\lambda_1(\Omega)$.}
\end{theorem}


\begin{remark}
\textnormal{In particular, if we choose $f(u)=u^{1+\alpha}$ ($\alpha> 0$) in Theorem 1.1, then
\begin{equation*}
F(u)=\int_u^{+\infty} \frac{d\tau}{\tau^{1+\alpha}}=\frac{1}{\alpha}u^{-\alpha}.
\end{equation*}
The Theorem 1.1 shows that the solution of (1.1) blows up if the follow condition (c1) is satisfied.\\
(c1). There exists a $\theta\in (0,1)$ such that for sufficiently large $t$,
\begin{equation*}
\frac{1}{\alpha}t^{\alpha}\leq t^{\frac{\theta}{m}},\quad \textrm{i.e.}, \quad t^{\frac{m\alpha-\theta}{m}}\leq \alpha,
\end{equation*}
where $m$ and $\alpha$ are positive constants.\\
It is easy to find that for sufficiently large $t$ and positive constants $m$, $\alpha$, $\theta$, we have
\begin{equation*}
\begin{split}
(c1)&\Longleftrightarrow
\left\{
\begin{array}{lc}
\exists\, \theta \in (0,1) \textrm{ such that } 0< m\alpha< \theta, \quad \text{for $\alpha<1$}\\
\exists\, \theta \in (0,1) \textrm{ such that } 0< m\alpha\leq \theta, \quad \text{for $\alpha\geq1$}
\end{array}
\right.\\
&\Longleftrightarrow 0<m\alpha <1.
\end{split}
\end{equation*}
}
\end{remark}

The assertions of Remark 1.1 leads to the following result, which was obtained by Lin and Wu in an earlier paper \cite{LW2}.

\begin{corollary}
\textnormal{
Let $G$ be a locally finite connected graph and have the volume growth of positive degree $m$.
If $0<m\alpha<1$, then the solution of the following semilinear heat equation
\begin{equation*}
\left\{
\begin{array}{lc}
u_t=\Delta u + u^{1+\alpha} &\, \text{in $(0,+\infty)\times V$}\\
u(0,x)=a(x) &\, \text{in $V$}
\end{array}
\right.
\end{equation*}
blows up for any bounded, non-negative and non-trivial initial value, where $\alpha>0$.
}
\end{corollary}


\begin{remark}
\textnormal{Let $f(u)=u^{1+\alpha}$ ($\alpha> 0$), the Theorem 1.2 shows that the solution of (1.2) blows up if
$$\sum_{x\in \Omega^\circ}\mu(x)a(x)\phi_1(x)\geq \lambda_1^{\frac{1}{\alpha}}.$$
}
\end{remark}


The remaining parts of this paper are organized as follows. In Section 2, we introduce some concepts, notations and known results which are essential to prove the main results of this paper. In Sections 3 and 4, we give the proofs of
Theorems 1.1 and 1.2 respectively.

\section{Preliminaries}

Throughout the paper, we write $y\sim x$ if $y$ is adjacent to $x$, and
allow the edges on the graphs to be weighted. Weights are given by a function $\omega: V \times V\rightarrow[0,\infty)$ satisfying $\omega_{xy}=\omega_{yx}$ and
$\omega_{xy}> 0$ if and only if $x\sim y$.
Furthermore, let $\mu: V\rightarrow \mathbb{R}^+$ be a positive measure on the vertices of the $G$. In this paper, we only focus on the graphs satisfying
$$D_\mu:=\max_{x\in V}\frac{m(x)}{\mu(x)}<\infty,$$
where $m(x):=\sum_{y\sim x}\omega_{xy}.$

\subsection{The Laplace operators on graphs}

A function on a graph is understood as a function defined on its vertex set.
We use the notation $C(V)$ to denote the set of real functions on $V$.
For any $1\leq p<\infty$, we denote  by
$$\ell^p(V,\mu)=\left \{h\in C(V):\sum_{x \in V} \mu(x)|h(x)|^p<\infty\right \}$$
the set of $\ell^p$ integrable functions on $V$ with respect to
the measure $\mu$.  For $p=\infty$, let
$$\ell^{\infty}(V,\mu)=\left\{h\in C(V):\sup_{x\in V}|h(x)|<\infty\right\}.$$

The integration of a function $h\in C(V)$ is defined by
$$\int_V hd\mu=\sum_{x\in V}\mu(x)h(x).$$

For any function $h\in C(V)$, the $\mu-$Laplacian $\Delta$ of $h$ is defined by
$$\Delta h(x)=\frac{1}{\mu(x)}\sum_{y\sim x}\omega_{xy}(h(y)-h(x)),$$
it can be checked that $D_\mu<\infty$ is equivalent to the $\mu-$Laplacian $\Delta$ being bounded on $\ell^p(V,\mu)$ for all $p\in [1,\infty]$ (see \cite{HAESELER}).

Given a finite subset $\Omega\subset V$, denote by $C(\Omega^\circ)$ the set of real functions on $\Omega^\circ$. For any function $h\in C(\Omega^\circ)$, the Dirichlet Laplacian $\Delta_\Omega$ on $\Omega^\circ$ is defined as follows : first extend $h$ to the whole $V$ by setting $h\equiv0$ outside
$\Omega^\circ$ and then set
$$\Delta_\Omega h=(\Delta h)|_{\Omega^\circ}.$$
Thus for any $x\in \Omega^\circ$, we have
\begin{equation*}
\Delta_\Omega h(x)=\frac{1}{\mu(x)}\sum_{y\sim x}\omega_{xy}(h(y)-h(x)),
\end{equation*}
where $h(y)=0$ whenever $y\notin \Omega^\circ$. A simple calculation shows that $-\Delta_\Omega$ is a
positive self-adjoint operator (see \cite{gri09,AW}). We arrange the eigenvalues of $-\Delta_\Omega$ in increasing order, i.e.
$\lambda_1(\Omega)\leq \lambda_2(\Omega)\leq\cdots \leq \lambda_N(\Omega)$, where $N=\#\Omega$.
Indeed, $\lambda_1(\Omega)$ is positive and its corresponding eigenfunction $\phi_1(x)$ can be chosen as $\phi_1(x)>0$ for all $x\in \Omega^\circ$ (see \cite{FC,CLC,gri09}). In this paper, $\phi_1(x)$ is normalized as $\sum_{x\in \Omega^\circ}\mu(x)\phi_1(x)=1$.

\subsection{The volume growth of graphs}

The connected graph can be endowed with its graph distance $d(x,y)$, i.e. the smallest number of edges of a path between two vertices $x$ and $y$, then we define balls $B(x,r)=\{y\in V:d(x,y)\leq r\}$ for any $r\geq0$. The volume of a subset $U$ of $V$ can be written as $V(U)$ and $V(U)=\sum_{x\in U}\mu(x)$, for convenience, we usually abbreviate $V\big(B(x,r)\big)$ by $V(x,r)$. In addition,
we say that a graph has the volume growth of positive degree $m$, if for all $x\in V$, $r\geq 0$,
\begin{equation*}
V(x,r)\leq c_0 r^m,
\end{equation*}
where $c_0$ are some positive constants.

\subsection{The heat kernel on graphs}

We say that a function $p:(0,+\infty)\times V \times V\rightarrow \mathbb{R}$ is a fundamental solution of the heat
equation $u_t=\Delta u$ on $G=(V,E)$, if for any bounded initial condition $u_0:V\rightarrow \mathbb{R}$, the function
$$u(t,x)=\sum_{y\in V}p(t,x,y)u_0(y) \quad (t>0, \, x\in V)$$
is differentiable in $t$ and satisfies the heat equation, and for any $x\in V$,
$\lim\limits_{t\rightarrow0^+}u(t,x)=u_0(x)$ holds.

For completeness, we recall some important properties of the heat kernel $p(t,x,y)$ on a locally finite graph, as follows:
\newtheorem{proposition}{\textbf{Proposition}} [section]
\begin{proposition} [see \cite{AW,RW}]
\textnormal{For $t,s>0$ and any $x,y\in V$, we have\\
(i) \, $p(t,x,y)=p(t,y,x)$,\\
(ii) \, $p(t,x,y)> 0$,\\
(iii) \, $\sum_{y\in V}\mu(y)p(t,x,y)\leq 1$,\\
(iv) \, $\partial_t p(t,x,y)=\Delta_xp(t,x,y)=\Delta_yp(t,x,y)$,\\
(v) \, $\sum_{z\in V}\mu(z)p(t,x,z)p(s,z,y)=p(t+s,x,y)$.}
\end{proposition}

Lin et al. \cite{LW1} utilized the volume growth condition to obtain a on-diagonal lower estimate of heat kernel on locally finite graphs for large time. We recall it bellow.

\begin{proposition}[see \cite{LW1}]
\textnormal{Assume that $G$ satisfies the volume growth, then for large enough $t$,
\begin{equation*}
p(t,x,x)\geq \frac{1}{4V(x,C_0 t\log t)},
\end{equation*}
where $C_0>2D_\mu e$.}
\end{proposition}

\section{Proof of Theorem 1.1}

As a preparation for the proof of Theorem 1.1, we first introduce two lemmas.


\newtheorem{lemma}{\textbf{Lemma}}[section]
\begin{lemma}
\textnormal{Let $G$ be a locally finite connected graph.
For any $T>0$, if $g$ is bounded in $V$, then for any $x\in V$,
$\sum_{y\in V} \mu(y)p(t,x,y)g(y)$
converges uniformly in $(0,T]$.}
\end{lemma}


\begin{proof}[Proof]
We begin with recalling a previous result which was obtained in \cite{AW}.
If $\Delta$ is a bounded operator, then we have
\begin{equation*}
\tag{3.1}
P_t g(x)=e^{t\Delta}g(x)=\sum_{k=0}^{+\infty}\frac{t^k \Delta^k}{k!}g(x)=\sum_{y\in V}\mu(y)p(t,x,y)g(y).
\end{equation*}

Since $g$ is bounded in $V$, we can assume $|g(x)|\leq A$ in $V$, then
\begin{equation*}
|\Delta g(x)|=\Bigg |\frac{1}{\mu(x)}\sum_{y\sim x}\omega_{xy}\big(g(y)-g(x)\big)\Bigg |\leq 2D_\mu A.
\end{equation*}

By iteration, we obtain for any $k\in\mathbb{N}$ and $x\in V$,
$$\big |\Delta^kg(x)\big |\leq2^kD_\mu^kA.$$

Thus for any $t\in (0,T]$ and $x\in V$,
$$\Bigg | \frac{t^k\Delta^k}{k!}g(x) \Bigg | \leq \Bigg | \frac{T^k\Delta^k}{k!}g(x) \Bigg |
\leq \frac{T^k}{k!}2^kD_\mu^kA.$$

In view of
$$\sum_{k=0}^{+\infty}\frac{T^k}{k!}2^kD_\mu^kA=Ae^{2D_\mu T}<\infty,$$
we deduce that $\sum_{k=0}^{+\infty}\frac{t^k \Delta^k}{k!}g(x)$
converges uniformly in $(0,T]$.

This completes the proof of Lemma 3.1.
\end{proof}


\begin{lemma}
\textnormal{Let $G$ be a locally finite connected graph.
For any $t>0$ and $x\in V$, if $g$ is bounded in $V$, then we have
\begin{equation*}
\tag{3.2}
\sum_{y\in V} \mu(y)\Delta p(t,x,y)g(y)=\sum_{y\in V} \mu(y)p(t,x,y)\Delta g(y).
\end{equation*}
}
\end{lemma}


\begin{proof}[Proof]
A direct computation yields
\begin{equation*}
\begin{split}
\sum_{y\in V} \mu(y)\Delta p(t,x,y)g(y)=&
\sum_{y\in V}\sum_{z\in V}\omega_{yz}\big(p(t,x,z)g(y)-p(t,x,y)g(y)\big)\\
=&\sum_{y\in V}\sum_{z\in V}\omega_{yz}p(t,x,z)g(y)-\sum_{y\in V}\sum_{z\in V}\omega_{yz}p(t,x,y)g(y)\\
=&\sum_{z\in V}\sum_{y\in V}\omega_{yz}p(t,x,y)g(z)-\sum_{y\in V}\sum_{z\in V}\omega_{yz}p(t,x,y)g(y)\\
=&\sum_{y\in V}\sum_{z\in V}\omega_{yz}p(t,x,y)g(z)-\sum_{y\in V}\sum_{z\in V}\omega_{yz}p(t,x,y)g(y)\\
=&\sum_{y\in V} \mu(y)p(t,x,y)\Delta g(y).
\end{split}
\end{equation*}

Note that the above summations can be exchanged, since
\begin{equation*}
\begin{split}
\sum_{y\in V}\sum_{z\in V}\big|\omega_{yz}p(t,x,y)g(z)\big|
&\leq \sum_{y\in V}\mu(y)p(t,x,y)\left(\sum_{z\in V}\frac{\omega_{yz}}{\mu(y)}\left|g(z)\right|\right)\\
&\leq D_\mu A.
\end{split}
\end{equation*}

The Lemma 3.2 is proved.
\end{proof}


\begin{proof}[Proof of Theorem 1.1.]
Suppose that there exists a non-negative solution $u=u(t,x)$ of (1.1) in $[0,+\infty)\times V$.
Since $a(x)$ is non-negative and not trivial in $V$, we can assume that $a(\nu)>0$ with $\nu\in V$.

Taking an arbitrary $T>0$, we put
\begin{equation*}
\tag{3.3}
J_T(s)=\sum_{x\in V}\mu(x)p(T-s,\nu,x)u(s,x) \quad (0\leq s< T).
\end{equation*}
Obviously, $J_T$ is continuous with respect to $s$.
Since $u$ is bounded, according to Lemma 3.1,
we know that $J_T$ exists even though $G$ is locally finite.


Firstly, we show that $J_T$ is positive for all $s\in [0,T)$.
In fact, since $u(s,\nu)$ is non-negative in $[0,T)$ and $f$ is non-negative in $[0,+\infty)$,
for all $0\leq s <T$, it follows that
\begin{equation*}
\tag{3.4}
\frac{\partial u}{\partial s}(s,\nu)-\Delta u(s,\nu)\geq 0.
\end{equation*}

Note that
\begin{equation*}
\begin{split}
\Delta u(s,\nu)&=\frac{1}{\mu(\nu)}\sum_{y\sim \nu}\omega_{\nu y}\big(u(s,y)-u(s,\nu)\big)\\
&\geq -\frac{1}{\mu(\nu)}\sum_{y\sim \nu}\omega_{\nu y}u(s,\nu)\\
&\geq -D_\mu u(s,\nu),
\end{split}
\end{equation*}
then the inequality (3.4) gives
\begin{equation*}
\frac{\partial u}{\partial s}(s,\nu)\geq -D_\mu u(s,\nu),
\end{equation*}
which, together with $a(\nu)=u(0,\nu)>0$, yields
\begin{equation*}
u(s,\nu)\geq u(0,\nu)\exp(-D_\mu s)>0, \quad s\in [0,T).
\end{equation*}

Hence, for all $0\leq s < T$, we have
$$\sum_{x\in V}u(s,x)>0.$$
In view of the fact that $p(T-s,\nu,x)$ is positive,
we obtain $J_T(s)$ is positive in $[0,T)$.


Secondly, we shall prove that $J_T$ is differentiable with respect to $s$ and satisfies the following equation
\begin{equation*}
\tag{3.5}
\frac{d}{ds}J_T(s)=\sum_{x\in V} \mu(x)p(T-s,\nu,x)f(u(s,x)).
\end{equation*}

Since $u$ is bounded, by Lemma 3.1, we know that $J_T$ is uniformly convergent. Note also that,
$f$ is continuous in $[0,+\infty)$, which and the boundedness of $u$ imply that $f(u)$ is bounded.
Following from this and Lemma 3.1, we obtain that $\frac{d}{ds}J_T$ is also uniformly convergent.
Hence we can exchange the order of summation and derivation, get
\begin{equation*}
\frac{d}{ds}J_T(s)
=\sum_{x\in V} \left( \mu(x)\frac{\partial}{\partial s}p(T-s,\nu,x)u(s,x)
+\mu(x)p(T-s,\nu,x)\frac{\partial}{\partial s}u(s,x) \right).
\end{equation*}

From the property of heat kernel and Lemma 3.1, Lemma 3.2, we have
\begin{equation*}
\begin{split}
\frac{d}{ds}J_T(s)
=&\sum_{x\in V} \left( - \mu(x)\Delta p(T-s,\nu,x)u(s,x)+\mu(x)p(T-s,\nu,x)
\Big( \Delta u(s,x)+ f(u(s,x)) \Big) \right)\\
=&-\sum_{x\in V} \mu(x)\Delta p(T-s,\nu,x)u(s,x)+
\sum_{x\in V} \mu(x)p(T-s,\nu,x)\Delta u(s,x)\\
&+\sum_{x\in V} \mu(x)p(T-s,\nu,x)f(u(s,x))\\
=&\sum_{x\in V} \mu(x)p(T-s,\nu,x)f(u(s,x)).
\end{split}
\end{equation*}


Thirdly, we need to show that
\begin{equation*}
\tag{3.6}
\frac{d}{ds}J_T\geq f(J_T).
\end{equation*}

Suppose $\sum_{i=1}^n k_i \leq 1$ and $k_i\in [0,1)$, then there exists $k^*\in [0,1)$ such that $\sum_{i=1}^n k_i+k^* = 1$. Since $f$ is convex in $[0,+\infty)$, using Jensen's inequality,
for any $x_1,x_2,\cdots,x_n\in [0,+\infty)$, we have
\begin{equation*}
k_1f(x_1)+k_2f(x_2)+\cdots +k_nf(x_n)+k^*f(0)\geq f(k_1x_1+k_2x_2+\cdots +k_nx_n+k^*\cdot 0).
\end{equation*}
In view of $f(0)=0$, we have
\begin{equation*}
k_1f(x_1)+k_2f(x_2)+\cdots +k_nf(x_n)\geq f(k_1x_1+k_2x_2+\cdots +k_nx_n).
\end{equation*}

Especially, since $f$ is continuous, if $\sum_{i=1}^{\infty}k_if(x_i)$ and $\sum_{i=1}^{\infty}k_ix_i$ converge,
we have
\begin{equation*}
\tag{3.7}
\sum_{i=1}^{\infty}k_if(x_i)\geq f\left(\sum_{i=1}^{\infty}k_ix_i \right),
\end{equation*}
where $\sum_{i=1}^{\infty} k_i \leq 1$, \, $(0\leq k_i< 1)$.

We have shown that $J_T$ and $\frac{d}{ds}J_T$ both are convergent, which combines with (3.7) and
\begin{equation*}
\sum_{x\in V}\mu(x)p(T-s,\nu,x)\leq 1
\end{equation*}
to obtain
\begin{equation*}
\sum_{x\in V} \mu(x)p(T-s,\nu,x)f(u(s,x))\geq f\left( \sum_{x\in V}\mu(x)p(T-s,\nu,x)u(s,x) \right),
\end{equation*}
this is the desired inequality (3.6).


Next, we consider the following function:
\begin{equation*}
\begin{split}
Q(s)&= F(J_T(0))-F(J_T(s))\\
&=\int_{J_T(0)}^{J_T(s)}\frac{1}{f(\tau)}d\tau
\quad\quad\quad (0\leq s< T).
\end{split}
\end{equation*}
It is not difficult to find that this function is well defined, since
$J_T(s) > 0$ for any $s\in [0,T)$ and $f(\tau)> 0$ for all $\tau> 0$.

We observe from (3.6) that for any $s\in [0,T)$,
\begin{equation*}
Q'(s)= J'_T(s) \frac{1}{f(J_T(s))}\geq 1.
\end{equation*}

Owing to $Q(0)=0$ and using the Mean-value theorem, for any $0<\varepsilon<T$, we get
\begin{equation*}
\tag{3.8}
Q(T-\varepsilon)\geq T-\varepsilon.
\end{equation*}

Since $f(\tau)> 0$ for all $\tau> 0$ and $J_T(s) > 0$ for any $s\in [0,T)$,
we conclude that
$F(J_T(s))$ is positive for all $s\in [0,T)$.
Furthermore, we deduce from inequality (3.8) that
\begin{equation*}
F(J_T(0))> F(J_T(0))-F(J_T(T-\varepsilon))=Q(T-\varepsilon) \geq T-\varepsilon.
\end{equation*}
Letting $\varepsilon\rightarrow 0$, we obtain
\begin{equation*}
\tag{3.9}
F\left(\sum_{x\in V}\mu(x)p(T,\nu,x)a(x)\right)\geq T.
\end{equation*}


From the given condition $V(x,r)\leq c_0 r^m$ $(c_0>0,r\geq 0,m>0)$ and the Proposition 2.2,
we have for large enough $T$,
\begin{equation*}
\tag{3.10}
p(T,\nu,\nu)\geq \frac{1}{4c_0C_0^m}\left(T\log T\right)^{-m}.
\end{equation*}

Hence, for sufficiently large $T$, we have
\begin{equation*}
\tag{3.11}
\begin{split}
&\sum_{x\in V}\mu(x)p(T,\nu,x)a(x)\\
\geq &\mu(\nu)p(T,\nu,\nu)a(\nu)\\
\geq &\widetilde{C}(T\log T)^{-m},
\end{split}
\end{equation*}
where $\widetilde{C}=\frac{\mu(\nu)a(\nu)}{4c_0C_0^m}>0$.

Let us come back to the inequality (3.9), which, together with (3.11) and the fact that
$F$ is non-increasing,
one obtains
\begin{equation*}
\tag{3.12}
F\left(\widetilde{C}(T\log T)^{-m}\right)\geq T
\end{equation*}
for large enough $T$.

On the other hand, it is easy to observe from the limit
$$\lim_{T \rightarrow +\infty}\frac{\widetilde{C}^{-\frac{1}{m}}\log T}{T^\delta}=0 \quad \quad
(\delta>0)$$
that
$$\widetilde{C}^{-\frac{1}{m}}\log T < T^\delta \quad \quad (\delta>0)$$
for sufficiently large $T$.
So, we can choose a $T_1>0$ such that
$$\widetilde{C}^{-\frac{1}{m}}\log T< T^{\frac{1-\theta}{\theta}} \quad \quad (0<\theta<1)$$
for all $T>T_1$. Since $f(\tau)>0$ for $\tau>0$, we find that $F$ is strictly decreasing
in $(0,+\infty)$.
Thus we have
\begin{equation*}
\tag{3.13}
F\left(\widetilde{C}(T\log T)^{-m}\right)< F\left( T^{-\frac{m}{\theta}}\right)
\end{equation*}
for all $T>T_1$.

The given condition (1.3) in Theorem 1.1 shows that there exist a $T_2\in (0,+\infty)$ and a $\theta\in (0,1)$ such that for $T>T_2$,
\begin{equation*}
\tag{3.14}
F\left( T^{-\frac{m}{\theta}}\right)\leq T.
\end{equation*}

Combining (3.13) and (3.14), we obtain
\begin{equation*}
F\left(\widetilde{C}(T\log T)^{-m}\right)< T
\end{equation*}
for all $T> \max\{T_1, T_2\}$. However, this contradicts (3.12).

The proof of Theorem 1.1 is complete.
\end{proof}

\section{Proof of Theorem 1.2}

In order to prove Theorem 1.2 we need the following two lemmas


\begin{lemma}[Strong Maximum Principle]
\textnormal{Let $G=(V,E)$ be a locally finite connected graph and $\Omega\subset V$ be finite.
For any $T>0$, we assume that $v(t,x)$ is bounded and continuous with respect to $t$ in $[0,T)\times \Omega$, which satisfies
\begin{equation*}
\tag{4.1}
\begin{cases}
\frac{\partial}{\partial t}v(t,x)-\Delta_\Omega v(t,x) -k(t,x)v(t,x)\geq 0, &\quad (t,x)\in (0,T)\times \Omega^\circ,\\
v(0,x)\geq 0, &\quad x\in \Omega^\circ,\\
v(t,x)\geq 0, &\quad (t,x)\in [0,T)\times \partial\Omega,
\end{cases}
\end{equation*}
where $k(t,x)$ is bounded in $[0,T)\times \Omega^\circ$. Then, $v(t,x)\geq 0$ in $[0,T)\times \Omega$.
}
\end{lemma}


\begin{proof}[Proof]
Set $\mathcal{L}v= v_t -\Delta_\Omega v-kv$.
Then, we have $\mathcal{L}v\geq 0$ in $(0,T)\times \Omega^\circ$.
Let $(t_0,x_0)$ be a minimum point of function $v$ in $(0,T)\times \Omega^\circ$.
In the following, we shall prove $v(t_0,x_0)\geq0$ by contradiction.

Assume that $v(t_0,x_0)<0$.
Case 1. If $k(t_0,x_0)<0$, we have
\begin{equation*}
\mathcal{L}v(t_0,x_0)<v_t(t_0,x_0)-\Delta_\Omega v(t_0,x_0).
\end{equation*}
Note that the function $t\mapsto v(t,x_0)$ in $(0,T)$ attains its minimum at $t=t_0$, we obtain
\begin{equation*}
\tag{4.2}
v_t(t_0,x_0)\leq 0
\end{equation*}
(if $t_0<T$, then $v_t(t_0,x_0)=0$). Hence, we conclude that
\begin{equation*}
\tag{4.3}
\mathcal{L}v(t_0,x_0) < -\Delta_\Omega v(t_0,x_0).
\end{equation*}

On the other hand, the function $x\mapsto v(t_0,x)$ in $\Omega^\circ$ attains the minimum at $x=x_0$, thus
\begin{equation*}
\tag{4.4}
\Delta_\Omega v(t_0,x_0)=\frac{1}{\mu(x_0)}\sum_{y\sim x_0}\omega_{x_0y}\big(v(t_0,y)-v(t_0,x_0)\big)\geq 0.
\end{equation*}
Substituting (4.4) into (4.3), we obtain
$\mathcal{L}v(t_0,x_0)<0$,
which is a contradiction with $\mathcal{L}v\geq 0$ in $(0,T)\times \Omega^\circ$.

Case 2. If $k(t_0,x_0)\geq0$,
since $k(t,x)$ is bounded in $[0,T)\times \Omega^\circ$,
we perform a transformation $\psi=v e^{-l t}$\,$\big(l>k(t_0,x_0)\big)$, then we get
\begin{eqnarray*}
\mathcal{L}v(t_0,x_0) &= & v_t(t_0,x_0)-\Delta_\Omega v(t_0,x_0)-k(t_0,x_0)v(t_0,x_0)\\
&\xlongequal {v =e^{l t}\psi}& e^{l t_0}\left( \psi_t(t_0,x_0)-\Delta_\Omega\psi(t_0,x_0)-(k(t_0,x_0)-l)\psi(t_0,x_0) \right)\\
&=& e^{l t_0} \widetilde{\mathcal{L}}\psi(t_0,x_0),
\end{eqnarray*}
where $\widetilde{\mathcal{L}}\psi \equiv \psi_t-\Delta_\Omega\psi-(k-l)\psi$.

Owing to $k(t_0,x_0)-l<0$ and $\psi(t_0,x_0)=v(t_0,x_0)e^{-lt_0}<0$, we have
\begin{equation*}
\tag{4.5}
\begin{split}
\widetilde{\mathcal{L}}\psi(t_0,x_0)&<\psi_t(t_0,x_0)-\Delta_\Omega \psi(t_0,x_0)\\
&=e^{-lt_0}\left(v_t(t_0,x_0)-\Delta_\Omega v(t_0,x_0)\right).
\end{split}
\end{equation*}
Combining (4.2),(4.4) and (4.5), we obtain $\widetilde{\mathcal{L}}\psi(t_0,x_0)<0$.
So, $\mathcal{L}v(t_0,x_0)=e^{l t_0} \widetilde{\mathcal{L}}\psi(t_0,x_0)<0$.
This contradicts with $\mathcal{L}v\geq 0$ in $(0,T)\times \Omega^\circ$.

Hence, $v(t_0,x_0)\geq 0$. Furthermore, since $v(t,x)$ in $(0,T)\times \Omega^\circ$ attains its minimum at $(t_0,x_0)$,
we have $v(t,x)\geq 0$ in $(0,T)\times \Omega^\circ$. In addition, from (4.1) we find
$v(t,x)\geq 0$ in $(\{0\}\times\Omega^\circ)\cup ([0,T)\times \partial\Omega).$
Thus $v(t,x)\geq 0$ in $[0,T)\times \Omega$.

The proof of Lemma 4.1 is complete.

\end{proof}


\begin{lemma}[Comparison Principle]
\textnormal{Let $G=(V,E)$ be a locally finite connected graph and $\Omega\subset V$ be finite.
For any $T>0$, we assume that $u(t,x)$ and $\underline{u}(t,x)$ are bounded and continuous with respect to $t$ in $[0,T)\times \Omega$, which satisfy
\begin{equation*}
\tag{4.6}
\begin{cases}
u_t-\Delta_\Omega u -g(u)\geq \underline{u}_t-\Delta_\Omega \underline{u} -g(\underline{u}), &\quad (t,x)\in (0,T)\times \Omega^\circ,\\
u(0,x) \geq \underline{u}(0,x), &\quad x\in \Omega^\circ,\\
u(t,x) \geq \underline{u}(t,x), &\quad (t,x)\in [0,T)\times \partial\Omega,
\end{cases}
\end{equation*}
where $g\in C^1[M_1,M_2]$ and $M_1\leq \underline{u},u \leq M_2$ for any $(t,x)\in [0,T)\times \Omega$.
Then, $u\geq \underline{u}$ in $[0,T)\times \Omega$.
}
\end{lemma}


\begin{proof}[Proof]
Set $v=u-\underline{u}$. Then
$$v_t-\Delta_\Omega v \geq g(u)-g(\underline{u})$$
in $(0,T)\times \Omega^\circ$.
Define a function by
\begin{equation*}
k(t,x)=
\left\{
\begin{array}{lc}
\frac{ g(u(t,x))- g(\underline{u}(t,x))}{u(t,x)-\underline{u}(t,x)} &\quad \text{for $(t,x)$ such that $u(t,x)\neq \underline{u}(t,x)$}\\
0 &\quad \text{for $(t,x)$ such that $u(t,x)=\underline{u}(t,x)$}.
\end{array}
\right.
\end{equation*}
Since $M_1\leq \underline{u},u \leq M_2$
for any $(t,x)\in [0,T)\times \Omega$ and $g\in C^1[M_1,M_2]$,
by the Mean-value theorem, we deduce that $k(t,x)$ is bounded in $[0,T)\times \Omega$.

Moreover, it follows from the definition of $k(t,x)$ that $g(u)-g(\underline{u})=kv$.
Hence, for any $(t,x)\in (0,T)\times \Omega^\circ$,
$$v_t-\Delta_\Omega v \geq kv.$$

Thus, the assertion of the Lemma 4.2 follows immediately from the Lemma 4.1.

\end{proof}


\begin{proof}[Proof of Theorem 1.2.]
We consider the function
\begin{equation*}
\tag{4.7}
J(t)=\sum_{x\in \Omega^\circ}\mu(x)u(t,x)\phi_1(x) \quad \quad (t\geq0),
\end{equation*}
where $\phi_1$ is a eigenfunction corresponding to the first eigenvalue $\lambda_1(\Omega)$
(see Section 2).
It is clear that $J(0)\equiv \kappa$,
$J(t)$ is continuous with respect to $t$ and well-defined on the existence interval of the solution $u$.


Firstly, we show that $J(t)$ is positive for all $t\in [0,+\infty)$.

In fact,
from the Lemma 4.2 and the assumptions in Theorem 1.2 related to $f$, we find $u(t,x)\geq 0$ in $[0,\infty)\times \Omega$
when $a(x)$ is not negative in $\Omega^\circ$.

Set $\Omega_1:=\{x\in \Omega^\circ:a(x)> 0\}$. Since $a(x)$ is not trivial in $\Omega^\circ$,
we get $\Omega_1\neq \varnothing$.
For any $z\in \Omega_1$, due to the fact that
$f$ is non-negative in $[0,+\infty)$, we have
\begin{equation*}
\tag{4.8}
\frac{\partial u}{\partial t}(t,z)-\Delta_\Omega u(t,z)\geq 0 \quad \quad (t\geq0).
\end{equation*}
By elementary calculus, we deduce that
\begin{equation*}
\tag{4.9}
\begin{split}
\Delta_\Omega u(t,z)
&\geq -\frac{1}{\mu(z)}\sum_{y\sim z}\omega_{z y}u(t,z)\\
&\geq -D_\mu u(t,z).
\end{split}
\end{equation*}
Combining (4.8) and (4.9), we verify that
\begin{equation*}
\tag{4.10}
\frac{\partial u}{\partial t}(t,z)\geq -D_\mu u(t,z).
\end{equation*}
The inequality (4.10) implies
\begin{equation*}
u(t,z)\geq a(z)\exp(-D_\mu t)>0  \quad \quad (t\geq0).
\end{equation*}

Hence, for any $t\geq 0$, we have
$$\sum_{x\in \Omega^\circ}u(t,x)>0.$$
Since $\mu(x)$ and $\phi_1(x)$ are positive in $\Omega^\circ$,
we conclude $J(t)>0$ for $t\in [0,+\infty)$.\\


Secondly, we prove that
\begin{equation*}
J'(t)\geq -\lambda_1J(t)+f(J(t)).
\end{equation*}

By using the fact that $\Delta_\Omega$ is self-adjoint, we have
\begin{equation*}
\tag{4.11}
\begin{split}
J'(t)=&\sum_{x\in \Omega^\circ}\mu(x)u_t(t,x)\phi_1(x)\\
=&\sum_{x\in \Omega^\circ}\mu(x)\phi_1(x)\Delta_\Omega u(t,x)+\sum_{x\in \Omega^\circ}\mu(x)\phi_1(x)
f(u(t,x))\\
=&\sum_{x\in \Omega^\circ}\mu(x)\Delta_\Omega \phi_1(x)u(t,x)+\sum_{x\in \Omega^\circ}\mu(x)\phi_1(x)f(u(t,x))\\
=&-\lambda_1\sum_{x\in \Omega^\circ}\mu(x)\phi_1(x)u(t,x)+\sum_{x\in \Omega^\circ}\mu(x)\phi_1(x)
f(u(t,x)).
\end{split}
\end{equation*}
Since $\sum_{x\in \Omega^\circ}\mu(x)\phi_1(x)=1$, by Jensen's inequality, we obtain
\begin{equation*}
\tag{4.12}
\sum_{x\in \Omega^\circ}\mu(x)\phi_1(x)f(u(t,x))\geq f\left(
\sum_{x\in \Omega^\circ}\mu(x)\phi_1(x)u(t,x)\right)=f(J(t)).
\end{equation*}
Combining (4.11) with (4.12) gives
\begin{equation*}
\tag{4.13}
J'(t)\geq -\lambda_1J(t)+f(J(t)).
\end{equation*}


Thirdly, we claim that there exists a positive constant $K>\kappa$\, such that $f(\tau)>2\lambda_1\tau$ for
$\tau\geq K.$

We prove this assertion by contradiction. Assume that for any $K>\kappa$, there exists a $\tau^*\geq K$
such that $$f(\tau^*)\leq 2\lambda_1\tau^*.$$

Based on the assumption above,
given a fixed $K_0>\kappa$,
there exists a $\tau_0^*\geq K_0$ such that $$f(\tau_0^*)\leq 2\lambda_1\tau_0^*.$$
In addition,
for any $K> \tau_0^*> \kappa$, there exists a $\tau_k^*\geq K$
such that $$f(\tau_k^*)\leq 2\lambda_1\tau_k^*.$$

Since $f(\tau)$ is convex in $[0,+\infty)$, for any $\tau\in \left(\tau_0^*,\tau_k^*\right)$,
we have
\begin{equation*}
\begin{split}
f(\tau)&\leq \frac{\tau_k^*-\tau}{\tau_k^*-\tau_0^*}f(\tau_0^*)+\frac{\tau-\tau_0^*}{\tau_k^*-\tau_0^*}f(\tau_k^*)\\
&\leq 2\lambda_1\tau_0^*\cdot\frac{\tau_k^*-\tau}{\tau_k^*-\tau_0^*}+ 2\lambda_1\tau_k^*\cdot\frac{\tau-\tau_0^*}{\tau_k^*-\tau_0^*}\\
&=2\lambda_1\tau.
\end{split}
\end{equation*}
In view of $\tau_0^*< K\leq\tau_k^*$, we get
\begin{equation*}
f(\tau)\leq 2\lambda_1\tau,
\quad \quad \tau\in(\tau_0^*,K).
\end{equation*}
Because $K$ is an arbitrary number that is greater than $\tau_0^*$, we obtain for all $\tau\in (\tau_0^*,+\infty)$,
\begin{equation*}
\tag{4.14}
f(\tau)\leq 2\lambda_1\tau.
\end{equation*}

Since $f$ is positive in $(\tau_0^*,+\infty)$, the inequality (4.14) implies
\begin{equation*}
\frac{1}{f(\tau)}\geq \frac{1}{2\lambda_1\tau}
\end{equation*}
for all $\tau\in (\tau_0^*,+\infty)$, which contradicts with the fact that $\frac{1}{f}$ is integrable at
$\tau\rightarrow \infty$.
Hence, there exists $K>\kappa$ such that $f(\tau)>2\lambda_1\tau$ for
$\tau\geq K$.


Now, we are in a position to prove the assertion of Theorem 1.2.

We note that $f(\tau)$ and $f(\tau)-\lambda_1 \tau$ are positive in $(\kappa,+\infty)$, thus
\begin{equation*}
\tag{4.15}
\frac{1}{f(\tau)-\lambda_1\tau}<\frac{2}{f(\tau)} \quad \quad (\tau\geq K).
\end{equation*}
Since $\int^{+\infty}_K\frac{d\tau}{f(\tau)}<+\infty$, the inequality (4.15) implies
\begin{equation*}
\int_{K}^{+\infty}\frac{d\tau}{f(\tau)-\lambda_1 \tau}< +\infty.
\end{equation*}
This yields
\begin{equation*}
\tag{4.16}
\int_{J(0)}^{+\infty}\frac{d\tau}{f(\tau)-\lambda_1 \tau}
=\int_{J(0)}^{K}\frac{d\tau}{f(\tau)-\lambda_1 \tau}
+\int_{K}^{+\infty}\frac{d\tau}{f(\tau)-\lambda_1 \tau} <+\infty.
\end{equation*}
It is clear that there is no singularity in the above integral,
since $J(0)\equiv\kappa>0$ and
$f(\tau)$, $f(\tau)-\lambda_1 \tau$ are positive in $[\kappa,+\infty)$.

On the other hand,
it follows from (4.13) that
\begin{equation*}
\tag{4.17}
\int_{J(0)}^{J(t)}\frac{dJ}{f(J)-\lambda_1 J}\geq t.
\end{equation*}

Letting $t\rightarrow +\infty$ in (4.17) leads to a contradiction with (4.16), which
together with the property of continuation on the positive classical solution yields that
$J(t)\rightarrow +\infty$ as $t\rightarrow T^-$.

This completes the proof of Theorem 1.2.
\end{proof}

\section*{Acknowledgments}
This research is supported by the National Science Foundation of China (Grant No.11671401).

\def\refname{\Large

\vspace{10pt}

{\setlength{\parindent}{0pt}
Yong Lin,\\
Department of Mathematics, Renmin University of China,  Beijing, 100872, P. R. China\\
linyong01@ruc.edu.cn\\
Yiting Wu,\\
Department of Mathematics, Renmin University of China,  Beijing, 100872, P. R. China\\
yitingly@126.com}

\end{document}